\documentclass[11pt,reqno]{amsart}
\usepackage{amsmath,amssymb,mathtools}

\usepackage{amssymb,mathtools,calc,verbatim,enumitem,tikz,url,mathrsfs,fullpage}
\usepackage[noadjust]{cite}
\usepackage{bbm}
\usepackage{stmaryrd}
\usepackage{textcomp}
\usepackage{setspace}
\usepackage{amsthm}
\usepackage{amsmath}
\usepackage{graphicx}
\usepackage{marvosym}
\usepackage{empheq}
\usepackage{latexsym}
\usepackage[T1]{fontenc}
\usepackage{color}
\usepackage[noadjust]{cite}

\usepackage[
  colorlinks=false,
  citebordercolor={0.55 0.80 1},
  linkbordercolor={1 0.65 0.65},
  urlbordercolor={0.55 0.80 1},
  pdfborder={0 0 1}
]{hyperref}
\usepackage{cleveref}
\usepackage{dsfont}

\usepackage{todonotes}

\theoremstyle{plain}
\newtheorem{theorem}{Theorem}[section]
\newtheorem{lemma}[theorem]{Lemma}
\newtheorem{proposition}[theorem]{Proposition}
\newtheorem{corollary}[theorem]{Corollary}

\newtheorem{claim}[theorem]{Claim}

\theoremstyle{definition}

\newtheorem{question}[theorem]{Question}

\theoremstyle{remark}

\newcommand{\RT}{\mathrm{RT}}

\title[Prescribed-order subdigraphs with large minimum out-degree]{Prescribed-order subdigraphs with large minimum out-degree}

\author{Bin Chen}
\address{School of Mathematics and Statistics, Fuzhou University, Fujian, China}
\email{cbfzu03@163.com}

\author{Lanchao Wang}
\address{School of Mathematics, Nanjing University, Nanjing, China, and ECOPRO, Institute for Basic Science, 55 Expo-ro, Yuseong-gu, Daejeon, 34126, Korea}
\email{lanchaowang@foxmail.com}

\date{}

\begin{document}

\begin{abstract}
Alon introduced $d(s)$ as the largest integer $d$ such that every digraph on $2n$ vertices with minimum out-degree at least $s$ contains a subdigraph on $n$ vertices with minimum out-degree at least $d$. He proved that $s/2-d(s)=O(\sqrt{s\log s})$, and further asked whether this deficit can be bounded by an absolute constant. Steiner answered this question in the negative by constructing suitable tournaments, and showed that $s/2-d(s)=\Omega(\log s)$.

Using a different construction, we show that the deficit grows at least
on the square-root scale, rather than merely logarithmically, improving the best known lower bound due to Steiner from $\Omega(\log s)$ to $\Omega(\sqrt{s})$ and leaving only a factor of $\sqrt{\log s}$ between the lower and upper bounds. This also completely settles a question raised by Steiner for tournament hosts. More generally, in the broader setting considered by Alon, our construction applies whenever the prescribed subdigraphs contain any fixed positive proportion of the vertices of the host digraph rather than specifically one half.
\end{abstract}

\maketitle

\section{Introduction}

In 2006, Alon \cite{Alon} introduced the parameter $d(s)$, defined as
the largest integer $d$ such that, for every positive integer $n$,
every digraph on $2n$ vertices with minimum out-degree at least $s$
contains a subdigraph on $n$ vertices with minimum out-degree at least
$d$. Since a random set of $n$ vertices is expected to
retain approximately half of the out-neighbours of each vertex, the
natural benchmark for $d(s)$ is $s/2$. By applying the probabilistic method,
Alon \cite{Alon} proved that
$
d(s)\ge \frac{s}{2}-O\bigl(\sqrt{s\log s}\bigr),
$
and he further posed the following question.

\begin{question}[\cite{Alon}]\label{n=2m}
Determine the value of $d(s)$. In particular, is it true that
$
\frac{s}{2}-d(s)=O(1)?
$
\end{question}

This question belongs to the broader study of vertex partitions under
degree constraints. Classical results for graphs include the
decomposition theorem of Lov\'asz \cite{LovaszDecomp} and the
minimum-degree partition theorem of Stiebitz \cite{Stiebitz}. Related
questions have also been studied for digraphs; see, for example,
\cite{Yang} for balanced tournament bisections,
\cite{Majority,GKP} for majority colourings, and \cite{LLS} for
judicious partitions.

Surprisingly, Steiner \cite{Steiner} answered Question~\ref{n=2m} in the negative. Concretely, he
constructed a sequence of tournaments on $2n$ vertices, each with
minimum out-degree $s=n-1$, such that every subtournament on $n$ vertices has
minimum out-degree at most
$
\frac{s}{2}-\left(\frac12+o(1)\right)\log_3 s.
$
This indicates that
$
\frac{s}{2}-d(s)=\Omega(\log s),
$
which means that the deficit is unbounded by an absolute constant even if the host digraph is
restricted to a tournament.

Motivated by this construction, Steiner also introduced the parameter
$d_T(s)$, defined as the largest integer $d$ such that, for every
positive integer $n$, every tournament on $2n$ vertices with minimum
out-degree at least $s$ contains a subtournament on $n$ vertices with
minimum out-degree at least $d$. Notice that
$
d_T(s)\ge d(s).
$
In \cite{Steiner}, Steiner then posed the following question.

\begin{question}[\cite{Steiner}]\label{question:tournament}
Is it true that
$
\frac{s}{2}-d_T(s)=\Theta(\sqrt{s})?
$
\end{question}

The upper bound $\frac{s}{2}-d_T(s)=O(\sqrt{s})$ follows from a result
of Alon, Bang-Jensen and Bessy \cite[Theorem~7.1]{ABB}. Thus, it remained to
establish a matching lower bound. Steiner \cite{Steiner} took the
first step by proving a logarithmic deficit.

In this paper, we prove
$\frac{s}{2}-d_T(s)=\Omega(\sqrt{s})$ using a random regular tournament
construction, thereby resolving Question~\ref{question:tournament} in
the affirmative.

\begin{theorem}\label{thm}
$
\frac{s}{2}-d_T(s)=\Theta(\sqrt{s}).
$
\end{theorem}

Moreover, Theorem~\ref{thm} improves the lower bound for Alon's original problem,
Question~\ref{n=2m}, from $\Omega(\log s)$ to $\Omega(\sqrt{s})$.
That is, in general digraphs, the deficit must grow at least on the square-root scale, rather than
merely logarithmically, leaving only a factor of
$\sqrt{\log s}$ between the currently known lower and upper bounds.

\begin{corollary}\label{cor:main}
$
\Omega(\sqrt{s})
\le
\frac{s}{2}-d(s)
\le
O\bigl(\sqrt{s\log s}\bigr).
$
\end{corollary}

Alon \cite{Alon} also asked the more general question of characterizing
the quadruples $(n,m,s,d)$ for which every digraph on $n$ vertices
with minimum out-degree at least $s$ contains a subdigraph on $m$
vertices with minimum out-degree at least $d$. To study the fixed-ratio setting, we fix a real number $\alpha>1$ and
let $d(\alpha,s)$ denote the largest integer $d$ such that for every
positive integer $n$, every digraph on $\lfloor\alpha n\rfloor$
vertices with minimum out-degree at least $s$ contains a subdigraph on $n$ vertices with minimum out-degree at least $d$. Analogously, we define
$d_T(\alpha,s)$ for tournament hosts. Clearly, we have
$
d(2,s)=d(s)
$, $
d_T(2,s)=d_T(s),
$
as well as
$
d_T(\alpha,s)\ge d(\alpha,s).
$

The natural benchmark in this setting is $s/\alpha$, since a uniformly
chosen $1/\alpha$-fraction of the vertex set retains, in expectation,
approximately a $1/\alpha$-fraction of every out-neighbour. By a straightforward adaptation of Alon's probabilistic argument
\cite{Alon}, one obtains $
\frac{s}{\alpha}-d(\alpha,s)
=
O_\alpha\bigl(\sqrt{s\log s}\bigr).
$
Recently, Chen \cite{Chen} proved that
$
\frac{s}{\alpha}-d(\alpha,s)
=
\Omega_\alpha(\log s)
$ when $\alpha$ is an integer at least two. 

We improve this logarithmic lower bound to the square-root scale for
every real $\alpha>1$. Together with our extension of the out-splitting
theorem of Alon, Bang-Jensen and Bessy \cite{ABB}, this yields the
following strengthening of Theorem~\ref{thm}.
\begin{theorem}\label{thm:general-alpha}
For every real number $\alpha>1$, 
$
\frac{s}{\alpha}-d_T(\alpha,s)
=
\Theta_\alpha(\sqrt{s}).
$
\end{theorem}

We remark that our proof of the upper bound differs from that of Alon, Bang-Jensen and Bessy \cite{ABB}. In particular, our argument gives an alternative proof of their result. Combining Theorem~\ref{thm:general-alpha} with Alon's upper bound
\cite{Alon} gives the following immediate corollary.

\begin{corollary}\label{cor:general-alpha}
For every real number $\alpha>1$, 
$
\Omega_\alpha(\sqrt{s})
\le
\frac{s}{\alpha}-d(\alpha,s)
\le
O_\alpha\bigl(\sqrt{s\log s}\bigr).
$
\end{corollary}

The rest of the paper is organized as follows. In the next section, we give an overview of the proof of Theorem~\ref{thm:general-alpha}. In Section~3, we prove Theorem~\ref{thm:general-alpha}, first establishing
the lower bound on the deficit and then the corresponding upper bound.


\section{Overview of the proof of Theorem~\ref{thm:general-alpha}}\label{sec:proof-overview}

The proof of Theorem~\ref{thm:general-alpha} has two main ingredients. For the lower bound on the deficit, we shall consider a uniformly
random labelled regular tournament and then use McKay's enumeration of
regular tournaments \cite{McKay} to show that, with positive
probability, none of its prescribed-size subtournaments has minimum
out-degree too close to the natural benchmark.

For the upper bound,  we will follow the discrepancy
framework of Alon, Bang-Jensen and Bessy \cite{ABB}, but replace their
entropy-based partial-colouring argument and final application of
Spencer's theorem by an iterative application of the Lovett--Meka
partial-colouring method \cite{LM}. This gives the required control for
every prescribed density and, in particular, provides a new proof of
the balanced case of their out-splitting theorem.

Let $T$ be a tournament. For a set $X\subseteq V(T)$, we write $T[X]$ for the subtournament induced by $X$.  Denote by $N_T^+(v)$ the set of out-neighbours of $v$, and by
$d_T^+(v)$ the corresponding out-degree in $T$. Moreover, let $\delta^+(T)$ denote the minimum out-degree of $T$.

\subsection{The lower bound on the deficit}

Fix $\alpha>1$, let $M=\alpha N+O_\alpha(1)$ be an odd integer, and pick a
uniformly random labelled regular tournament $T$ on $M$ vertices. For a
fixed vertex set $X$ of size $N$, we estimate the probability that
$\delta^+(T[X])\ge (N-1)/2-\gamma\sqrt N$.

The estimate is organized around out-degree vectors, which record the
out-degree of every vertex. Let $F_m$ denote the maximum, over all integer
vectors $\sigma\in\mathbb Z^m$, of the number of labelled tournaments on
$m$ vertices having out-degree vector $\sigma$. 
A vertex-deletion recurrence shows that
$F_m\le O(1)^m2^{\binom m2}m^{-(m-2)/2}$. If $T[X]$ has minimum
out-degree at least $h=(N-1)/2-\gamma\sqrt N$, then its out-degree vector
$\mathbf{a}=(a_x)_{x\in X}$ satisfies $a_x\ge h$ and
$\sum_{x\in X}a_x=\binom N2$. Hence there are at most $(C\gamma\sqrt N)^N$ possible vectors
$\mathbf a$.

Fix such a vector, set $Y=V(T)\setminus X$, and write $Q=|Y|$.
There are at most $F_N$ choices for $T[X]$. Regularity determines the
out-degree of each vertex of $X$ into $Y$, so there are at most
$C_\alpha^N2^{NQ}N^{-N/2}$ possible orientations of the cross-edges. Once these edges
are fixed, the out-degree vector of $T[Y]$ is forced, leaving at most $F_Q$
choices. Comparing this with McKay's enumeration of regular
tournaments \cite{McKay}, we can deduce that
\[
\Pr\left(\delta^+(T[X])\ge\frac{N-1}{2}-\gamma\sqrt N\right)
\le N(C_\alpha\gamma)^N.
\]
A union bound, with $\gamma$ chosen sufficiently small, gives a regular tournament in which
every subtournament on $N$ vertices has minimum out-degree less than
$(N-1)/2-\gamma\sqrt N$.

Finally, choose $N$ as small as possible such that
$\lfloor\alpha N\rfloor\ge2s+2$. Then
$N=2s/\alpha+O_\alpha(1)$. By adjusting the order by at most one vertex to handle parity, we
obtain a tournament on $\lfloor\alpha N\rfloor$ vertices with minimum
out-degree at least $s$. It follows that
$d_T(\alpha,s)\le s/\alpha-a_\alpha\sqrt s$ for some $a_\alpha>0$.

\subsection{The upper bound on the deficit}
For the upper bound, the main ingredient is a prescribed-order rounding statement. Let
$\mathcal A=(A_1,\ldots,A_m)$ be an indexed family of non-empty subsets
of an $M$-element set $\Omega$ satisfying the growth condition
$
|\{i:|A_i|\le t\}|\le 3t
$
for every $t\ge1$. We prove that, for every
$q\in\{0,\ldots,M\}$, there is a set $X\subseteq\Omega$ of size exactly
$q$ such that
\[
\left||X\cap A_i|-\frac qM|A_i|\right|
=O(\sqrt{|A_i|}),
\]
where $i\in[m]$.

To encode the choice of $X$, we use a vector in $[-1,1]^\Omega$, where
the values $1$ and $-1$ correspond respectively to selected and
unselected elements. We start from the constant vector $\mathbf{z}$ whose coordinates are all equal to $2q/M-1$. We call a coordinate unfixed if its value lies in $(-1,1)$. Then we iteratively apply the Lovett--Meka partial-colouring lemma to the unfixed coordinates.
At each step, at least half of the unfixed coordinates are rounded to
values in $\{-1,1\}$. Meanwhile, the total coordinate sum is preserved and
the discrepancy on every $A_i$ is controlled. Once only a bounded number
of coordinates remain, they are rounded directly. This produces a vector
$\mathbf z^*\in\{-1,1\}^\Omega$ with coordinate sum $2q-M$ and
discrepancy $O(\sqrt{|A_i|})$ on each $A_i$. Setting
$
X=\{x\in\Omega:z_x^*=1\},
$
we obtain the required set of size $q$.

For a tournament $T$ on $M$ vertices, we apply this conclusion to the
family
\[
A_v:=N_T^+(v),
\]
where $v\in V(T)$ and we omit empty out-neighbourhoods. The required growth condition follows
from the fact that any tournament has at most $2t+1$ vertices of
out-degree at most $t$. Therefore, for every $q$, there is a set
$X\subseteq V(T)$ of size $q$ such that
\[
\left|
|N_T^+(v)\cap X|-\frac qM d_T^+(v)
\right|
\le O\left(\sqrt{d_T^+(v)}\right)
\]
for each $v\in V(T)$. Taking $M=\lfloor\alpha n\rfloor$ and $q=n$,
and applying this bound to a tournament with minimum out-degree at least
$s$, one can get
$
d_T(\alpha,s)\ge s/\alpha-O(\sqrt s).
$

\section{Proof of Theorem~\ref{thm:general-alpha}}

\subsection{The lower bound: a random regular tournament construction}\label{sec:construction}

We prove the following result.

\begin{proposition}\label{thm:construction}
For every real number $\alpha>1$, there are constants $a_\alpha>0$ and $s_0(\alpha)$ such that, for every $s\ge s_0(\alpha)$,
\[
d_T(\alpha,s)\le \frac{s}{\alpha}-a_\alpha\sqrt{s}.
\]
\end{proposition}

We begin with a uniform bound on the number of tournaments having a prescribed out-degree vector. For a labelled tournament $T$ on $[m]$, set $\sigma(T):=(d_T^+(i))_{i\in[m]}$. For $\sigma\in\mathbb Z^m$, let $T_m(\sigma)$ be the number of labelled tournaments on $[m]$ with out-degree vector $\sigma$, and set $F_m=\max_\sigma T_m(\sigma)$.

\begin{lemma}\label{lem:score-fibre}
There is an absolute constant $A>0$ such that, for every $m\ge1$,
\[
F_m\le A^m2^{\binom m2}m^{-(m-2)/2}.
\]
\end{lemma}

\begin{proof}
We first prove  that there is an absolute constant $C>0$ such that
\[
        F_m\le C\frac{2^{m-1}}{\sqrt m}F_{m-1}
\]
for every $m\ge2$.

Fix $\sigma=(\sigma_1,\ldots,\sigma_m)$  such that $T_m(\sigma)=F_m$.  Observe that the vertex $m$ has exactly $\sigma_m$
out-neighbours in $[m-1]$.  Let $S$ be the out-neighbourhood of $m$.
Obviously, we have $S\subseteq[m-1]$ and $|S|=\sigma_m$. It is clear that the induced
tournament on $[m-1]$ has out-degree vector $\sigma^S$ with
\[
        \sigma_i^S=
        \begin{cases}
        \sigma_i, & i\in S,\\
        \sigma_i-1, & i\notin S.
        \end{cases}
\]
Therefore, using the standard upper bound for binomial coefficients,
we obtain
\[
        F_m
        \le \sum_{\substack{S\subseteq[m-1]\\ |S|=\sigma_m}}
        T_{m-1}(\sigma^S)
        \le \binom{m-1}{\sigma_m}F_{m-1}
        \le C\frac{2^{m-1}}{\sqrt m}F_{m-1}.
\]
From the fact that $F_1=1$ one can deduce that
\[
        F_m\le C^{m-1}2^{\binom m2}(m!)^{-1/2}.
\]
Since $m!\ge (m/e)^m$, we have $(m!)^{-1/2}\le e^{m/2}m^{-m/2}$.
As a consequence, we are able to find a large $A$ satisfying that
\[
        F_m\le A^m2^{\binom m2}m^{-m/2}
        \le A^m2^{\binom m2}m^{-(m-2)/2},
\]
as desired. \end{proof}

Let $\RT(M)$ denote the number of labelled regular tournaments on $M$ vertices. McKay's asymptotic formula immediately produces the following estimate.

\begin{theorem}[McKay \cite{McKay}]\label{thm:mckay}
There is an absolute constant $B>0$ such that, for every odd integer $M$,
\[
\RT(M)\ge B^{-M}2^{\binom M2}M^{-(M-2)/2}.
\]
\end{theorem}

The following lemma estimates the probability that a fixed vertex set induces an almost regular subtournament of a random regular tournament.

\begin{lemma}\label{lem:fixed-set}
Fix $\alpha>1$. There is a constant $C_{\alpha}>0$ such that the following holds. Let $M$ be odd, let $|M-\alpha N|\le 1$, and let $R$ be a uniformly random labelled regular tournament on $M$ vertices. For every fixed $N$-set $X\subseteq V(R)$ and every fixed $\gamma>0$,
\[
\Pr\left(\delta^+(R[X])\ge \frac{N-1}{2}-\gamma\sqrt N\right)\le N(C_{\alpha}\cdot\gamma)^N
\]
for all sufficiently large $N$.
\end{lemma}

\begin{proof}
Set $Y=V(R)\setminus X$ and $Q=|Y|$. Let $h=(N-1)/2-\gamma\sqrt N$ and suppose that $\delta^+(R[X])\ge h$. If $a=(a_i)_{i\in X}$ is the out-degree vector of $R[X]$, then $a_i\ge h$ and $\sum_{i\in X}a_i=\binom N2$. Write $z_i$ for $a_i-\lceil h\rceil$. Note that $z_i\ge0$ and
\[
\sum_{i\in X}z_i=\binom N2-N\lceil h\rceil\le \gamma N^{3/2}.
\]
Hence, for sufficiently large $N$, the number of possible out-degree vectors $a$ is at most
\[
\binom{\lfloor\gamma N^{3/2}\rfloor+N}{N}\le (C_1\gamma\sqrt N)^N,
\]
 where $C_1$ is an absolute constant.

Fix such a vector $a$. There are at most $F_N$ choices for $R[X]$. For each $i\in X$, regularity forces $i$ to have exactly $(M-1)/2-a_i$ out-neighbours in $Y$. Thus, the number of possible orientations between $X$ and $Y$ is at most
\[
\prod_{i\in X}\binom{Q}{(M-1)/2-a_i}\le \binom{Q}{\lfloor Q/2\rfloor}^{N}\le C_2^N2^{NQ}Q^{-N/2}.
\]
Once $R[X]$ and the cross-edges are fixed, the out-degree vector of $R[Y]$ is determined, so there are at most $F_Q$ choices for $R[Y]$. By  Lemma~\ref{lem:score-fibre}, together with $\binom N2+NQ+\binom Q2=\binom M2$ as well as $Q=\Theta_{\alpha}(N)$, the number of regular tournaments satisfying the fixed vector $a$ is at most
\[
F_NF_QC_2^N2^{NQ}Q^{-N/2}\le C_{\alpha}^N2^{\binom M2}N^{-(M+N-4)/2}.
\]
 On the other hand, Theorem~\ref{thm:mckay} and $M=\alpha N+O_{\alpha}(1)$ gives
\[
\RT(M)\ge C_{\alpha}^{-N}2^{\binom M2}N^{-(M-2)/2}.
\]
Therefore, the probability that $R[X]$ has out-degree vector $a$ is
at most $C_{\alpha}^NN^{-(N-2)/2}$. Summing over all possible vectors $a$, we conclude that there exists a $C_{\alpha}$ such that
\[
\Pr\left(\delta^+(R[X])\ge h\right)\le (C_1\gamma\sqrt N)^NC_{\alpha}^NN^{-(N-2)/2}\le N(C_{\alpha}\gamma)^N,
\]
as desired.
\end{proof}

\begin{proof}[Proof of Proposition~\ref{thm:construction}]
Let $C_\alpha$ be the constant in Lemma~\ref{lem:fixed-set}. Choose $K_\alpha>0$ such that $\binom MN\le K_\alpha^N$ whenever $|M-\alpha N|\le1$ and $N$ is sufficiently large, and then choose $\gamma>0$ with $K_\alpha C_\alpha\gamma<1$.

Let $s$ be sufficiently large, and choose $N$ as small as possible such that $m:=\lfloor\alpha N\rfloor\ge2s+2$. Then $m=2s+O_\alpha(1)$ and $N=2s/\alpha+O_\alpha(1)$. Choose $\varepsilon\in\{0,1\}$ such that $M:=m+\varepsilon$ is an odd integer. Notice that $|M-\alpha N|\le1$.

Let $R$ be a uniformly random regular tournament on $M$ vertices. By Lemma~\ref{lem:fixed-set}, we get
\[
\begin{aligned}
\Pr\left(\exists X\subseteq V(R),\ |X|=N,\
\delta^+(R[X])\ge \frac{N-1}{2}-\gamma\sqrt N\right)\le N(K_\alpha C_\alpha\gamma)^N=o(1).
\end{aligned}
\]
Consequently, there is a regular tournament $R$ in which every $N$-set $X$ satisfies  $\delta^+(R[X])<(N-1)/2-\gamma\sqrt N$.

If $\varepsilon=0$, then we let $D=R$. If $\varepsilon=1$, then we let $D$ be the subtournament obtained from $R$ by deleting one vertex. In both cases, $|V(D)|=m=\lfloor\alpha N\rfloor$. If $\varepsilon=0$, then $m$ is odd and $\delta^+(D)=(m-1)/2\ge s$. If $\varepsilon=1$, then every vertex of $R$ has out-degree $m/2$, and hence $\delta^+(D)\ge m/2-1\ge s$ since deleting one vertex decreases each remaining out-degree by at
most one.

For every $N$-set $X\subseteq V(D)$, we have $D[X]=R[X]$.
Consequently,
$
\delta^+(D[X])
<
\frac{N-1}{2}-\gamma\sqrt N.
$ Since $N=2s/\alpha+O_\alpha(1)$, there is some $a_\alpha>0$ such that
\[
\frac{N-1}{2}-\gamma\sqrt N\le \frac{s}{\alpha}-a_\alpha\sqrt s
\]
for all sufficiently large $s$\;(for example, one may take $a_\alpha=(\gamma/4)\sqrt{2/\alpha}$). In summary, $D$ contains no subtournament on $N$ vertices of minimum out-degree at least $s/\alpha-a_\alpha\sqrt s$, completing the proof of the proposition.
\end{proof}

\subsection{The upper bound: an out-splitting theorem}\label{sec:deficit-upper}
We prove the following result.
\begin{proposition}\label{thm:deficit-bound}
For every fixed real number $\alpha>1$, there exists $s_0(\alpha)$ such that, for every $s\ge s_0(\alpha)$,
\[
d_T(\alpha,s)\ge \frac{s}{\alpha}-B\sqrt{s},
\]
where $B$ is the absolute constant from Lemma~\ref{thm:out-splitting}.
\end{proposition}
We first record the following exact version of the Lovett--Meka theorem
\cite[Theorem~2.1]{LM}, which follows from their approximate statement
by a compactness argument.

\begin{lemma}[Lovett--Meka]\label{lem:LM}
Let $N\ge1$ and $m\ge0$ be integers. Let
$\mathbf{u}_0,\ldots,\mathbf{u}_m$ be vectors in $\mathbb R^N$, let
$\mathbf{z}\in[-1,1]^N$, and let $c_0,\ldots,c_m$ be nonnegative real
numbers satisfying that
\[
\sum_{i=0}^me^{-c_i^2/16}\le N/16.
\]
Then there exists $\mathbf{z}'\in[-1,1]^N$ such that at least
$\lceil N/2\rceil$ coordinates of $\mathbf{z}'$ belong to
$\{-1,1\}$, and for every $0\le i\le m$,
\[
\left|\left\langle
\mathbf{z}'-\mathbf{z},\mathbf{u}_i
\right\rangle\right|
\le c_i\lVert\mathbf{u}_i\rVert_2.
\]
\end{lemma}

\begin{proof}
The initial version of Lovett and Meka gives an approximate form of the 
statement. Specifically, they verified that for every $\delta>0$, there is a
vector $\mathbf{z}'\in[-1,1]^N$ satisfying the same discrepancy
inequalities, and at least $\lceil N/2\rceil$ coordinates satisfy that
$
|z'_j|\ge1-\delta.
$

For each positive integer $k$, we employ their result by taking
$\delta=1/k$. This produces a vector
$\mathbf{z}^{(k)}\in[-1,1]^N$ such that, for every $0\le i\le m$,
\[
\left|\left\langle
\mathbf{z}^{(k)}-\mathbf{z},\mathbf{u}_i
\right\rangle\right|
\le c_i\lVert\mathbf{u}_i\rVert_2, 
\]
 and there are at least
$\lceil N/2\rceil$ coordinates satisfying
\[
|z_j^{(k)}|\ge1-\frac1k.
\]

Pick a set $I_k\subseteq[N]$ of size
$\lceil N/2\rceil$ which consists of such coordinates. Since there are only
finitely many subsets of $[N]$ of this size, by passing to a subsequence
we may assume that $I_k=I$ is independent of $k$.

Since $[-1,1]^N$ is compact, we may pass to a further subsequence such
that $\mathbf{z}^{(k)}$ converges to some
$\mathbf{z}'\in[-1,1]^N$. For every $j\in I$, we have
$|z_j^{(k)}|\ge1-1/k$, and hence $|z'_j|=1$. Thus at least
$\lceil N/2\rceil$ coordinates of $\mathbf{z}'$ belong to
$\{-1,1\}$.
By the continuity of the inner product, we conclude that, for every $0\le i\le m$,
\[
\left|\left\langle
\mathbf{z}'-\mathbf{z},\mathbf{u}_i
\right\rangle\right|
\le c_i\lVert\mathbf{u}_i\rVert_2.
\]

The lemma thus follows.
\end{proof}

We then present the following prescribed-size rounding
lemma, which says that for a family containing only linearly many sets of each bounded size, one can choose a set of any prescribed size while controlling all intersections up to a square-root error.
\begin{lemma}\label{thm:rounding}
There exists an absolute constant $B>0$ such that the following holds. Let $M$ and $m$ be positive integers, let $\Omega$ be a set of size $M$, and let
$
\mathcal A=(A_1,\ldots,A_m)
$
be an indexed family of non-empty subsets of $\Omega$ satisfying
$
\bigl|\{i\in[m]:|A_i|\le t\}\bigr|\le 3t
$
for every integer $t$ with $1\le t\le M$. Then, for every integer $q$ with $0\le q\le M$, there exists a set $X\subseteq\Omega$ such that $|X|=q$, and for every $i\in[m]$,
$$
\left||X\cap A_i|-\frac{q}{M}|A_i|\right|
\le B\sqrt{|A_i|}.
$$

\end{lemma}

\begin{proof}
Order the sets so that
$
a_1:=|A_1|\le\cdots\le a_m:=|A_m|.
$
The growth assumption implies
$
i\le3a_i
$
for each $i\in[m]$. Set $p=q/M$. The cases $q=0$ and $q=M$ are immediate, so we can assume $0<q<M$ and start from
$
\mathbf{z}^{(0)}=(2p-1)\mathbf{1}_\Omega.
$
Its coordinate sum is $2q-M$. We call a coordinate fixed once it takes a value in $\{-1,1\}$.

Suppose that the current vector is $\mathbf{z}\in[-1,1]^\Omega$, and let $W$ be the set of coordinates not yet in $\{-1,1\}$. If $|W|\ge64$, then we let
$
r=\lfloor |W|/64\rfloor
$
and $c_0=0$, and for every $i\in[m]$ we set
\[
c_i=
\begin{cases}
0,&i\le r,\\
4\sqrt{2\log(i/r)},&i>r.
\end{cases}
\]
By the definition of the $c_i$, we have
\[
\sum_{i=0}^me^{-c_i^2/16}
\le1+r+\sum_{i>r}(r/i)^2
\le1+2r
\le |W|/16.
\]
By applying Lemma~\ref{lem:LM} to the coordinates in $W$ with
\[
\mathbf{u}_0=\mathbf{1}_W
\quad\text{and}\quad
\mathbf{u}_i=\mathbf{1}_{A_i\cap W},
\]
 we can obtain that
$
\langle \mathbf{z}'-\mathbf{z},\mathbf{1}_W\rangle=0
$ as $c_0=0$ and $\mathbf{u}_0=\mathbf{1}_W$.
Hence, 
$
\sum_{x\in W}z'_x=\sum_{x\in W}z_x,
$
which means that the sum of the coordinates in $W$ is preserved. Meanwhile, at least half of the coordinates in $W$ become fixed in $\{-1,1\}$. We then remove these newly fixed coordinates from $W$ and apply Lemma~\ref{lem:LM} again to the remaining unfixed coordinates. Thus, if $n_j$ denotes the number of unfixed coordinates after the $j$th step, then
$
n_{j+1}\le \frac{n_j}{2}.
$
Since the coordinates that have already been fixed are never changed, and the sum over the unfixed coordinates is preserved at every step, the sum of all coordinates remains equal to its initial value $2q-M$ throughout the iteration.

We terminate the process when fewer than $64$ coordinates remain
unfixed. Let $F$ be the set of fixed coordinates. Then $|\Omega\setminus F|<64$. Since the sum over the
unfixed coordinates is preserved throughout the iteration,
$
R:=2q-M-\sum_{x\in F}z_x
   =\sum_{x\in \Omega\setminus F}z_x.
$
In particular, $|R|\le |\Omega\setminus F|$. Moreover,
$
R
\equiv M-|F|
=|\Omega\setminus F|
\pmod 2.
$
Thus $(|\Omega\setminus F|+R)/2$ is an integer between $0$ and $|\Omega\setminus F|$. We may
therefore assign the value $1$ to exactly $(|\Omega\setminus F|+R)/2$ coordinates
of $\Omega\setminus F$ and the value $-1$ to the remaining coordinates. This gives
a vector $\mathbf z^*\in\{-1,1\}^{\Omega}$ whose total coordinate sum
is $2q-M$.

\begin{claim}\label{clm:rounding-error}
For every $i\in[m]$, the sum of the absolute changes in
\[
\sum_{x\in A_i}z_x
\]
during the Lovett--Meka iterations is $O(\sqrt{a_i})$.
\end{claim}

\begin{proof}[Proof of the Claim]
Fix $i\in[m]$ and consider a stage at which $W$ is the set of unfixed
coordinates. If $c_i=0$, then Lemma~\ref{lem:LM} gives
$
\left|
\left\langle
\mathbf{z}'-\mathbf{z},\mathbf{1}_{A_i\cap W}
\right\rangle
\right|
\le
c_i\left\|\mathbf{1}_{A_i\cap W}\right\|_2
=0,
$
so this stage makes no contribution.

Suppose now that $c_i>0$. Then
$
3a_i\ge i>r=\lfloor |W|/64\rfloor,
$
and hence $|W|<192a_i$. Moreover, for some absolute constant $C>0$,
\[
c_i
=
4\sqrt{2\log\left(\frac{i}{r}\right)}
\le
C\sqrt{\log\left(\frac{Ca_i}{|W|}\right)}.
\]
Thus the absolute change at this stage is at most
\begin{align*}
c_i\left\|\mathbf{1}_{A_i\cap W}\right\|_2
&=
c_i\sqrt{|A_i\cap W|}\\
&\le
C\sqrt{\min\{a_i,|W|\}
\log\left(\frac{Ca_i}{|W|}\right)}.
\end{align*}

If $a_i\le |W|<192a_i$, this is $O(\sqrt{a_i})$. Since the
number of unfixed coordinates is at least halved at each step, there are
at most $8$ such stages.

For the remaining stages, $|W|<a_i$. If
$
2^{-j-1}a_i\le |W|<2^{-j}a_i
$
for some $j\ge0$, then the contribution of this stage is at most
$
C\sqrt{a_i}\,2^{-j/2}\sqrt{j+1}.
$
There is at most one such stage for each $j$, and
\[
\sum_{j\ge0}2^{-j/2}\sqrt{j+1}<\infty.
\]
Summing over all stages proves the claim.

 The final assignment of values to the fewer than $64$ remaining
coordinates changes the sum over $A_i$ by only $O(1)$. Hence
$
\left|
\sum_{x\in A_i}(z_x^*-z_x^{(0)})
\right|
\le C\sqrt{a_i}
$
for some absolute constant $C$. \end{proof}

Set
$
X=\{x\in\Omega:z_x^*=1\}.
$
As the total coordinate sum is $2q-M$, we get $|X|=q$, and
$
2\left(|X\cap A_i|-pa_i\right)
=
\sum_{x\in A_i}(z_x^*-z_x^{(0)}).
$ From Claim~\ref{clm:rounding-error} we know that
\[
\left||X\cap A_i|-pa_i\right|
=O(\sqrt{a_i}),
\]
which completes the proof of the lemma.\end{proof}

The tournament structure supplies exactly the growth condition required above.

\begin{lemma}\label{thm:out-splitting}
There exists an absolute constant $B>0$ such that the following holds. Let $M$ be a positive integer, let $T$ be a tournament on $M$ vertices, and let $q$ be an integer with $0\le q\le M$. Then there exists a set $X\subseteq V(T)$ such that $|X|=q$ and
$$
\left|
|N_T^+(v)\cap X|-\frac{q}{M}d_T^+(v)
\right|
\le B\sqrt{d_T^+(v)}
\quad \text{for every $v\in V(T)$.}$$

\end{lemma}

\begin{proof}
For any integer $t\ge1$, set
$
U_t=\{v\in V(T):d_T^+(v)\le t\}
$. The average out-degree of $T[U_t]$ is
$(|U_t|-1)/2$, while every vertex of $U_t$ has out-degree at most $t$ in
$T[U_t]$. Hence
$
\frac{|U_t|-1}{2}\le t,
$
and therefore $|U_t|\le2t+1\le3t$.

Thus the indexed family
$
\mathcal A=
\bigl(N_T^+(v):v\in V(T),\ d_T^+(v)>0\bigr)
$
satisfies the hypothesis of Lemma~\ref{thm:rounding}. Applying that
lemma, we obtain a set $X\subseteq V(T)$ of size $q$ such that
\[
\left|
|N_T^+(v)\cap X|-\frac{q}{M}d_T^+(v)
\right|
\le B\sqrt{d_T^+(v)}
\]
for every vertex $v$ of positive out-degree. If $d_T^+(v)=0$, then
both terms inside the absolute value are zero, so the same inequality
holds trivially.
\end{proof}

\begin{proof}[Proof of Proposition~\ref{thm:deficit-bound}]
Let $T$ be a tournament on $M=\lfloor\alpha n\rfloor$ vertices with $\delta^+(T)\ge s$, and we denote $p=n/M$. By Lemma~\ref{thm:out-splitting}, there is an $n$-set $X$ such that, for every vertex $v$, $$|N_T^+(v)\cap X|\ge p d_T^+(v)-B\sqrt{d_T^+(v)}.$$  Since $p\ge1/\alpha$, the function $x\mapsto px-B\sqrt x$ is increasing on $[s,\infty)$ whenever $s\ge B^2\alpha^2/4$. Thus, for every $v\in X$, we have
\[
d_{T[X]}^+(v)\ge ps-B\sqrt s\ge \frac{s}{\alpha}-B\sqrt s.
\]
It follows that $\delta^+(T[X])\ge s/\alpha-B\sqrt s$.
\end{proof}

\begin{proof}[Proof of Theorem~\ref{thm:general-alpha}]
Propositions~\ref{thm:construction} and~\ref{thm:deficit-bound}
together prove the theorem.
\end{proof}

We remark that Corollaries~\ref{cor:main} and~\ref{cor:general-alpha} follow immediately from Theorem~\ref{thm:general-alpha} and the upper bound due to Alon~\cite{Alon}.

\section*{Acknowledgements}

B. Chen was supported by the NSFC under grant numbers 12501473, 12471336. L. Wang was supported by the National Key R\&D Program of China under grant number 2024YFA1013900, the NSFC under grant number 12471327, the China Scholarship Council, and the Institute for Basic Science (IBS-R029-C4). The authors used AI to assist in working out their ideas and improving the grammar and presentation of the manuscript. All mathematical arguments and proofs in the final manuscript were written by the authors.

\end{document}